\documentclass{amsart}

\usepackage{amsmath}
\usepackage{amssymb}
\usepackage{amsthm}
\usepackage[msc-links,abbrev]{amsrefs}
\usepackage{hyperref}
\usepackage[noabbrev,capitalize]{cleveref}
\usepackage{mathrsfs}
\usepackage{mathtools}
\usepackage{slashed}
\usepackage{enumitem}

\setenumerate{label=(\roman*)}

\DeclareMathOperator{\tr}{tr}
\DeclareMathOperator{\Vol}{Vol}
\DeclareMathOperator{\dvol}{dvol}

\DeclareMathOperator{\Ric}{Ric}

\DeclareMathOperator{\Rm}{Rm}

\DeclareMathOperator{\End}{End}

\DeclareMathOperator{\Met}{Met}

\newcommand{\defn}[1]{{\boldmath\bfseries#1}}

\newcommand{\oM}{\overline{M}}

\newcommand{\og}{\overline{g}}

\newcommand{\cg}{\widetilde{g}}

\newcommand{\cM}{\widetilde{M}}

\newcommand{\hg}{\widehat{g}}

\newcommand{\lp}{\langle}
\newcommand{\rp}{\rangle}
\newcommand{\lv}{\lvert}
\newcommand{\rv}{\rvert}

\newcommand{\mI}{\mathcal{I}}

\newcommand{\mL}{\mathcal{L}}

\newcommand{\mQ}{\mathcal{Q}}

\newcommand{\mY}{\mathcal{Y}}

\newcommand{\bR}{\mathbb{R}}

\def\sideremark#1{\ifvmode\leavevmode\fi\vadjust{\vbox to0pt{\vss
 \hbox to 0pt{\hskip\hsize\hskip1em
 \vbox{\hsize3cm\tiny\raggedright\pretolerance10000
 \noindent #1\hfill}\hss}\vbox to8pt{\vfil}\vss}}}

\newcommand{\suchthatcolon}{\mathrel{}:\mathrel{}}

\newtheorem{theorem}{Theorem}[section]
\newtheorem{proposition}[theorem]{Proposition}
\newtheorem{lemma}[theorem]{Lemma}
\newtheorem{corollary}[theorem]{Corollary}

\theoremstyle{definition}

\theoremstyle{remark}

\numberwithin{equation}{section}

\makeatletter
\@namedef{subjclassname@2020}{%
  \textup{2020} Mathematics Subject Classification}
\makeatother

\begin{document}

\title[Global obstructions to conformally Einstein six-manifolds]{Global obstructions to conformally Einstein metrics in dimension six}
\author{Jeffrey S. Case}
\address{Department of Mathematics \\ Penn State University \\ University Park, PA 16802}
\email{jscase@psu.edu}
\keywords{Einstein metric, global conformal invariant, global obstruction}
\subjclass[2020]{Primary 5C25; Secondary 53C18}
\begin{abstract}
 We present a global conformal invariant on closed six-manifolds which obstructs the existence of a conformally Einstein metric.
 We show that this obstruction is nontrivial and, up to multiplication by a constant, is the unique such invariant.
 This also gives rise to a (possibly trivial) diffeomorphism invariant which obstructs the existence of an Einstein metric.
 We also discuss global conformal invariants which obstruct the existence of a conformally Einstein metric on closed six-manifolds with infinite fundamental group.
\end{abstract}
\maketitle

\section{Introduction}
\label{sec:intro}

A fundamental problem in Riemannian geometry is to determine whether a given closed manifold admits an \defn{Einstein metric};
i.e.\ a Riemannian metric $g$ for which the Ricci curvature satisfies $\Ric_g=\lambda g$ for some constant $\lambda\in\bR$.
Much is known about this problem in dimension at most four:
Every closed surface admits an Einstein metric, while there are topological obstructions to the existence of an Einstein metric on a closed three- or four-manifold~\cite{Besse}.
There are no known obstructions in higher dimension, and indeed there are many constructions of Einstein metrics in higher dimensions~\cites{Bohm1998,Besse,Yau1978,ChenDonaldsonSun2015a,ChenDonaldsonSun2015b,ChenDonaldsonSun2015c,Tian2015,BoyerGalickiKollar2005}.

A natural approach to finding obstructions, if they exist, is to first seek global obstructions to the existence of an Einstein metric in a given conformal class.
One reason for this is that the space of \defn{global conformal invariants} --- i.e.\ conformally invariant integrals of complete contractions of the metric, its inverse, the Riemannian volume form, and covariant derivatives of the Riemann curvature tensor --- is classified~\cite{Alexakis2012}.
Another reason is that the Hitchin--Thorpe inequality~\cites{Hitchin1974,Thorpe1969} is quickly derived by considering global conformal invariants in dimension four:

The space of global conformal invariants of a closed, oriented Riemannian four-manifold is three-dimensional and spanned by
\begin{equation*}
 \int_M Q_4^g\,\dvol_g \qquad\text{and}\qquad \int_M \lv W_g^\pm\rv_g^2\,\dvol_g ,
\end{equation*}
where $Q_4 := -\frac{1}{6}\Delta R + \frac{1}{24}R^2 - \frac{1}{2}\lv E\rv^2$ is the fourth-order $Q$-curvature~\cite{Branson1995}, $E := \Ric - \frac{R}{4}g$ is the trace-free part of the Ricci tensor, and $W^+$ (resp.\ $W^-$) is the self-dual (resp.\ anti-self-dual) part of the Weyl tensor.
Each of $\int Q_4$, $\int \lv W^+\rv^2$, and $\int \lv W^- \rv^2$ is nonnegative if $(M^4,g)$ is Einstein.
In particular, if the diffeomorphism invariant $\mQ(M) := \sup_{g} \int Q_4 ^g\, \dvol_g$ is negative, then $M$ does not admit an Einstein metric.
Moreover, the Gauss--Bonnet--Chern and Hirzebruch signature formulas,
\begin{align*}
 8\pi^2\chi(M) & = \frac{1}{4}\int\lv W_g^+\rv_g^2 \, \dvol_g + \frac{1}{4}\int\lv W_g^-\rv_g^2 \, \dvol_g + \int Q_4^g \, \dvol_g , \\
 12\pi^2\tau(M) & = \frac{1}{4}\int\lv W_g^+\rv_g^2 \, \dvol_g - \frac{1}{4}\int\lv W_g^-\rv_g^2 \, \dvol_g ,
\end{align*}
respectively, imply that $\mQ(M) \leq 8\pi^2\bigl( \chi(M) \pm \frac{3}{2} \tau(M) \bigr)$.
In particular, if $\chi(M) < \frac{3}{2}\lv\tau(M)\rv$, then $M$ does not admit an Einstein metric~\cite{Thorpe1969}.
A more careful analysis reveals that if $\chi(M) = \frac{3}{2}\lv\tau(M)\rv$ and $M^4$ admits an Einstein metric, then $M$ is, up to taking a finite cover and reversing the orientation, either a four-torus or a K3 surface~\cite{Hitchin1974}.
Note that there are infinitely many closed, simply-connected four-manifolds with $\chi(M) > \frac{3}{2}\lv\tau(M)\rv$ which do not admit Einstein metrics~\cite{Lebrun1996}.

We now specialize to closed six-manifolds.
In this case, the space of global conformal invariants is four-dimensional~\cites{Alexakis2012,CaseLinYuan2016,BaileyEastwoodGraham1994,BaileyGover1995,Gover2001,FeffermanGraham2012} and spanned by the total integrals of
\begin{subequations}
 \label{eqn:dim6-gci}
 \begin{align}
  \label{eqn:dim6-Q} Q_6 & := -\Delta^2J + 4\Delta J^2 - 8\delta\bigl( P(\nabla J) \bigr) + 4\Delta\lv P\rv^2 \\
   \notag & \qquad - 8J^3 + 24J\lv P\rv^2 - 16\tr P^3 - 8\lp B, P\rp , \\
  \label{eqn:dim6-pos} L_1 & := 2\Delta\lv W\rv^2 - 48\nabla^l(W_{ijkl}C^{ijk}) \\
   \notag & \qquad + \lv\nabla W\rv^2 + 48P_t^s W_{sijk}W^{tijk} - 8J\lv W\rv^2 - 64\lv C\rv^2  , \\
  \label{eqn:dim6-Delta} L_2 & := \frac{1}{2}\Delta\lv W\rv^2 - 8\nabla^l(W_{ijkl}C^{ijk}) + 8P_t^s W_{sijk}W^{tijk} - 2J\lv W\rv^2 - 8\lv C\rv^2 , \\
  \label{eqn:dim6-cubic} L_3 & := W_{ij}{}^{kl} W_{kl}{}^{st} W_{st}{}^{ij} ,
 \end{align}
\end{subequations}
where $P_{ij}$ is the Schouten tensor, $J$ is its trace, $C_{ijk}$ is the Cotton tensor, $W_{ijkl}$ is the Weyl tensor, and $B_{ij}$ is the Bach tensor;
see \cref{sec:invariant} for their definitions.
Note that $Q_6$ is (the negative of) the critical \defn{$Q$-curvature}~\cite{Branson1995} in dimension six.
We denote by $\mQ(M,g) := \int Q^g\,\dvol_g$ and $\mL_j(M,g) := \int L_j^g\,\dvol_g$, $j \in \{ 1, 2, 3\}$, the corresponding global conformal invariants.

We have chosen the basis~\eqref{eqn:dim6-gci} to illustrate the signs taken by evaluating global conformal invariants on conformal classes of closed, Einstein six-manifolds.
For example, Equation~\eqref{eqn:dim6-pos} implies that the negativity of $\mL_1(M,g)$ obstructs the existence of an Einstein metric conformal to $g$.
Indeed, this is the only such obstruction, up to multiplication by a positive constant.

\begin{theorem}
 \label{invariant}
 Let $(M^6,g)$ be a closed Riemannian $6$-manifold.  Then
 \begin{equation}
  \label{eqn:invariant}
  \mL_1(M,g) = \int_M \Bigl( \lv\nabla W\rv^2 + 48E_s^tW_{tijk}W^{sijk} - 64\lv C\rv^2 \Bigr) \, \dvol 
 \end{equation}
 is conformally invariant, where $E := P - \frac{J}{6}g$ is the trace-free part of the Schouten tensor.
 If $\hg := e^{2u}g$ is Einstein, then $\mL_1(M,g) \geq 0$ with equality if and only if $\hg$ is locally symmetric.
 Moreover, if $\mI$ is a global conformal invariant which is nonnegative at every Einstein metric, then $\mI = a\mL_1$ for some constant $a \geq 0$.
\end{theorem}

The proof of \cref{invariant} is relatively simple.
We have already noted that~\eqref{eqn:invariant} is a global conformal invariant which is nonnegative at conformally Einstein metrics.
The uniqueness of $\mL_1$ follows by considering simple examples;
see \cref{sec:proof}.

Importantly, $\mL_1$ is a \emph{nontrivial} obstruction to the existence of an Einstein metric in a given conformal class.

\begin{proposition}
 \label{t2-times-s4}
 Let $(T^2 \times S^4 , g := dx^2 \oplus d\theta^2)$ be the Riemannian product of a flat two-torus and a round four-sphere of constant sectional curvature one.
 Then
 \begin{equation*}
  \mL_1(T^2 \times S^4, g) = -96\pi^2\Vol_{dx^2}(T^2) .
 \end{equation*}
 In particular, $g$ is not conformally Einstein.
\end{proposition}

\Cref{t2-times-s4} follows by direct computation;
see \cref{t2-times-s4-general} for a more general statement.
Work of Gover and Nurowski~\cite{GoverNurowski2006} implies that $g$ is not even locally conformally Einstein.

\Cref{invariant} motivates the introduction of the diffeomorphism invariant
\begin{equation*}
 \mL_1(M^6) := \sup \left\{ \mL_1(g) \suchthatcolon g \in \Met(M) \right\} ,
\end{equation*}
where $\Met(M)$ denotes the space of Riemannian metrics on a given closed six-manifold $M$.
This invariant gives a topological obstruction to the existence of an Einstein metric.

\begin{theorem}
 \label{topological_obstruction}
 Let $M^6$ be a closed six-manifold with $\mL_1(M) \leq 0$.
 If $\mL_1(M)=0$, assume additionally that $\pi_1(M)=\infty$ and the universal cover of $M$ is not diffeomorphic to $\bR^6$.
 Then $M$ does not admit an Einstein metric.
\end{theorem}

We are not presently aware of a closed six-manifold $M$ for which $\mL_1(M)$ is finite, let alone nonpositive.

Since the product metric on $S^1 \times S^5$ is locally conformally flat, it holds that $\mL_1(S^1 \times S^5) \geq 0$.
A scaling argument and \cref{t2-times-s4} together imply that $\mL_1(T^2 \times S^4) \geq 0$.
If equality holds in either case, then \cref{topological_obstruction} implies that the relevant manifold does not admit an Einstein metric.
Note that, by the Hitchin--Thorpe inequality, neither $S^1 \times S^3$ nor $T^2 \times S^2$ admits an Einstein metric.

The global conformal invariant $\mL_1(M,g)$ depends on the choice of Einstein metric on a given closed six-manifold, if one exists.
Additionally, Einstein metrics are not generally critical points of the functional $g \mapsto \mL_1(M,g)$.
See \cref{sec:example} for details.
These observations limit the applicability of $\mL_1$ to other questions involving Einstein manifolds.

The total integrals of~\eqref{eqn:dim6-gci} give additional obstructions to the existence of an Einstein metric with nonpositive scalar curvature in a given conformal class.

\begin{theorem}
 \label{nonpositive-invariant}
 Let $(M^6,g)$ be a closed Riemannian six-manifold.
 Then
 \begin{equation}
  \label{eqn:nonpositive-invariant}
  \begin{split}
   \mQ(M,g) & := \int_M \left( -\frac{40}{9}J^3 + 16J\lv E\rv^2 - 16\tr E^3 - 8\lp B, E\rp \right) \, \dvol_g , \\
   \mL_1(M,g) & := \int_M \Bigl( \lv\nabla W\rv^2 + 48E_s^tW_{tijk}W^{sijk} - 64\lv C\rv^2 \Bigr) \, \dvol_g , \\
   \mL_2(M,g) & := \int_M \left( 8E_s^tW_{tijk}W^{sijk} - \frac{2}{3}J\lv W\rv^2 - 8\lv C\rv^2 \right) \, \dvol_g
  \end{split}
 \end{equation}
 are conformally invariant.
 Suppose that $\hg := e^{2u}g$ is an Einstein metric with nonpositive curvature.
 Then
 \begin{enumerate}
  \item $\mQ(M,g) \geq 0$ with equality if and only if $\mY(M,g)=0$;
  \item $\mL_1(M,g) \geq 0$ with equality if and only if $\hg$ is locally symmetric; and
  \item $\mL_2(M,g) \geq 0$ with equality if and only if $\mY(M,g)=0$ or $g$ is locally conformally flat.
 \end{enumerate}
\end{theorem}

Here $\mY(M,g)$ is the \defn{Yamabe constant}~\cite{LeeParker1987};
i.e.\ the infimum of the total scalar curvature over all unit-volume metrics conformal to $g$.

Any convex linear combination of $\mQ(M,g)$, $\mL_1(M,g)$, and $\mL_2(M,g)$ is also nonnegative if $(M,g)$ is conformal to an Einstein manifold with nonpositive scalar curvature.
We do not presently know whether these are the only such invariants.

The motivation for \cref{nonpositive-invariant} is that, by Myers' Theorem~\cite{Myers1941}, the global conformal invariants $\mQ$, $\mL_1$, and $\mL_2$ obstruct the existence of an Einstein metric on a closed conformal six-manifold with infinite fundamental group.

\begin{corollary}
 \label{fund-group-invariant}
 Let $(M^6,g)$ be a closed Einstein six-manifold with $\pi_1(M) = \infty$.
 Then
 \begin{enumerate}
  \item $\mQ(M,g) \geq 0$ with equality if and only if $\mY(M,g)=0$;
  \item $\mL_1(M,g) \geq 0$ with equality if and only if $g$ is locally symmetric; and
  \item $\mL_2(M,g) \geq 0$ with equality if and only if $\mY(M,g)=0$ or $g$ is locally conformally flat.
 \end{enumerate}
\end{corollary}

\Cref{nonpositive-invariant,fund-group-invariant} are simple consequences of the formulas~\eqref{eqn:dim6-gci}.
In \cref{sec:example} we compute the invariants $\mQ$, $\mL_1$, and $\mL_2$ on Riemannian products of spaceforms to show that they all give nontrivial obstructions.

\Cref{nonpositive-invariant} motivates the introduction of the diffeomorphism invariants
\begin{align*}
 \mQ^{\leq0}(M) & := \sup \left\{ \mQ(M,g) \suchthatcolon g \in \Met(M) , \mY(M,g) \leq 0 \right\} , \\
 \mL_1^{\leq0}(M) & := \sup \left\{ \mL_1(M,g) \suchthatcolon g \in \Met(M) , \mY(M,g) \leq 0 \right\} , \\
 \mL_2^{\leq0}(M) & := \sup \left\{ \mL_2(M,g) \suchthatcolon g \in \Met(M) , \mY(M,g) \leq 0 \right\} .
\end{align*}
One can formulate an analogue of \cref{topological_obstruction} for the invariants $\mQ^{\leq0}(M)$, $\mL_1^{\leq0}(M)$, and $\mL_2^{\leq0}(M)$ on a closed six-manifold with infinite fundamental group.
We presently cannot show that any of these invariants is finite, let alone nonpositive.

Unlike the situation in dimension four, the space of global conformal invariants in dimension six does \emph{not} admit a basis of invariants which are nonnegative at Einstein six-manifolds with infinite fundamental group, let alone Einstein six-manifolds in general.
However, the Gauss--Bonnet--Chern Theorem~\cite{Chern1945} implies that any global conformal invariant is a linear combination of the Euler characteristic and the three linear independent invariants of \cref{nonpositive-invariant}.

\begin{theorem}
 \label{gbc}
 Let $(M^6,g)$ be a closed six-manifold.
 Then
 \begin{equation*}
  64\pi^3\chi(M) = -\mQ(M,g) - \frac{1}{3}\mL_1(M,g) + \frac{7}{6}\mL_2(M,g) + \mL_3(M,g) .
 \end{equation*}
\end{theorem}

\Cref{gbc} follows readily from a known expression~\cite{Graham2000} of the Gauss--Bonnet--Chern formula in terms of global conformal invariants.
It is presently unclear whether one can combine \cref{fund-group-invariant,gbc} to find a closed six-manifold which does not admit an Einstein metric.

This article is organized as follows.

In \cref{sec:invariant} we give a detailed account of the global conformal invariants determined by~\eqref{eqn:dim6-gci}.
We also prove \cref{gbc}.

In \cref{sec:algebra} we compute the scalar invariants~\eqref{eqn:dim6-gci} at products of spaceforms.

In \cref{sec:proof} we prove \cref{invariant,fund-group-invariant,nonpositive-invariant,topological_obstruction}.

In \cref{sec:example} we discuss some examples, and in particular prove \cref{t2-times-s4}.

\section{Global conformal invariants in dimension six}
\label{sec:invariant}

Let $(M^n,g)$ be a Riemannian $n$-manifold.
The \defn{Schouten tensor} $P_{ij}$, the \defn{Weyl tensor} $W_{ijkl}$, the \defn{Cotton tensor} $C_{ijk}$, and the \defn{Bach tensor} $B_{ij}$ are
\begin{align*}
 P_{ij} & := \frac{1}{n-2}\left(R_{ij} - Jg_{ij}\right), \\
 W_{ijkl} & := R_{ijkl} - P_{ik}g_{jl} - P_{jl}g_{ik} + P_{il}g_{jk} + P_{jk}g_{il}, \\
 C_{ijk} & := \nabla_iP_{jk} - \nabla_jP_{ik}, \\
 B_{ij} & := \nabla^s C_{sij} + W_{isjt}P^{st} ,
\end{align*}
respectively, where $R_{ij}$ is the Ricci tensor, $R_{ijkl}$ is the Riemann curvature tensor, and $J := P_s{}^s$.
We use Penrose's abstract index notation throughout this article.

The importance of these tensors stems from their symmetries and conformal transformation properties.
The Weyl tensor is conformally invariant, and hence the Schouten tensor determines the behavior of the Riemann curvature tensor under conformal transformation.
Under conformal change of metric, the Cotton and Bach tensors depend only on the one-jet of the conformal factor~\cite{FeffermanGraham2012}.
The relevant algebraic symmetries of these tensors are
\begin{align*}
 P_{[ij]} & = 0 , \\
 W_{ijkl} & = W_{[ij][kl]} = W_{[kl][ij]} , & W_{[ijk]l} & = 0 , \\
 C_{ijk} & = C_{[ij]k}, & C_{[ijk]} & = 0 , \\
 B_{[ij]} & = 0 ,
\end{align*}
where square brackets indicate skew symmetrization.
All traces of these tensors can be computed from these symmetries and the identities
\begin{align*}
 J & = P_s{}^s, & W_{isj}{}^s & = 0, & C_{si}{}^s & = 0 .
\end{align*}
We also require the consequence
\begin{equation}
 \label{eqn:weyl-bianchi}
 \nabla_{[i} W_{jk]}{}^{lm} = -2C_{[ij}{}^{[l}g_{k]}{}^{m]}
\end{equation}
of the second Bianchi identity.

We now verify that the total integrals of~\eqref{eqn:dim6-gci} form a basis for the space of global conformal invariants in dimension six.

\begin{proposition}
 \label{gci}
 The space of global conformal invariants is four-dimensional and spanned by the total integrals of~\eqref{eqn:dim6-gci}.
\end{proposition}

\begin{proof}
 This follows from the following three observations.
 First, any global conformal invariant is a linear combination of $\int Q^g\,\dvol_g$ and an integral of a local conformal invariant~\cite{Alexakis2012}.
 Second, in dimension six,
 \begin{align*}
  I_1 & = \lv\nabla W\rv^2 - 16\nabla^l(W_{ijkl}C^{ijk}) + 16P_t^sW_{sijk}W^{tijk} - 32\lv C\rv^2 , \\
  I_2 & = W_i{}^k{}_j{}^l W_k{}^s{}_l{}^t W_s{}^i{}_t{}^j , \\
  I_3 & = W_{ij}{}^{kl} W_{kl}{}^{st} W_{st}{}^{ij} .
 \end{align*}
 give a basis for the space of local conformal invariants~\cites{FeffermanGraham2012,BaileyEastwoodGraham1994,Gover2001,BaileyGover1995}.
 Third, \eqref{eqn:weyl-bianchi} implies that, in dimension six,
 \begin{multline*}
  \frac{1}{2}\Delta\lv W\rv^2 = \lv\nabla W\rv^2 - 8\nabla^l(W_{ijkl}C^{ijk}) + 2J\lv W\rv^2 + 8P_t^s W_{sjkl} W^{tjkl} \\
   - 24\lv C\rv^2 - 4I_2 - I_3 . \qedhere
 \end{multline*}
\end{proof}

We conclude with the Gauss--Bonnet--Chern theorem in dimension six.

\begin{proof}[Proof of \cref{gbc}]
 It is known~\cite{Graham2000}*{pg.\ 38} that
 \begin{equation*}
  64\pi^3\chi(M) = \int_M \left( -Q_6^g - \frac{1}{2}I_1 + \frac{2}{3}I_2 + \frac{7}{6}I_3 \right) \, \dvol_g .
 \end{equation*}
 The conclusion readily follows.
\end{proof}

\section{Some computations on locally symmetric spaces}
\label{sec:algebra}

In this section we collect some basic computations on locally symmetric spaces, and especially products of spaceforms.

We begin by discussing contractions of certain algebraic curvature operators.
To that end, given elements $S_{ij}$ and $T_{ij}$ of $S^2T_p^\ast M$, denote by
\begin{align*}
 (S \wedge T)_{ijkl} & := S_{ik}T_{jl} + S_{jl}T_{ik} - S_{il}T_{jk} - S_{jk}T_{il} , \\
 (S \circ T)_{ij} & := S_i{}^u T_{uj} ,
\end{align*}
the Kulkarni--Nomizu product and composition with respect to the identification $S^2T_p^\ast M \cong \End(T_pM)$, respectively.
Given elements $S_{ijkl}$ and $T_{ijkl}$ of $S^2\Lambda^2 T_p^\ast M$ --- so that $S_{ijkl} = S_{[ij][kl]} = S_{[kl][ij]}$ --- denote by
\begin{equation*}
 (S \circ T)_{ijkl} := \frac{1}{2}S_{ij}{}^{uv} T_{uvkl}
\end{equation*}
the composition with respect to the identification $S^2\Lambda^2T_p^\ast M \cong \End(\Lambda^2T_p^\ast M)$.
We have the following identities for compositions of certain Kulkarni--Nomizu products.

\begin{lemma}
 \label{compositions}
 Let $(M^n,g)$ be a Riemannian manifold and let $S,T \in S^2T_p^\ast M$ be such that $S\circ T = 0$.
 Then
 \begin{align*}
  (S\wedge S)\circ(S\wedge S) & = 2S^2\wedge S^2 , \\
  (S\wedge S)\circ(S\wedge T) & = 0 , \\
  (S\wedge S)\circ(T\wedge T) & = 0 , \\
  (S\wedge T)\circ(S\wedge T) & = S^2\wedge T^2 .
 \end{align*}
\end{lemma}

\begin{proof}
 Given $S_1,S_2,S_3,S_4 \in S^2T_p^\ast M$, it is straightforward to check that
 \begin{equation*}
  (S_1\wedge S_2)\circ(S_3\wedge S_4) = (S_1\circ S_3)\wedge (S_2 \circ S_4) + (S_1\circ S_4)\wedge(S_2\circ S_3) .
 \end{equation*}
 The conclusion readily follows.
\end{proof}

Define the partial trace $\tr \colon S^2\Lambda^2T_p^\ast M \to S^2T_p^\ast M$ by
\begin{equation*}
 (\tr S)_{ij} := S_{iuj}{}^u .
\end{equation*}
We have the following identities for the trace of the Kulkarni--Nomizu products of elements of $S^2T_p^\ast M$.

\begin{lemma}
 \label{traces}
 Let $(M^n,g)$ be a Riemannian manifold and let $S, T \in S^2T_p^\ast M$ be such that $S\circ T = 0$.
 Then
 \begin{align*}
  \tr(S\wedge S) & = 2(\tr S)S - 2S^2 , \\
  \tr(S\wedge T) & = (\tr S)T + (\tr T)S . 
 \end{align*}
\end{lemma}

\begin{proof}
 Given $S,T \in S^2T_p^\ast M$, we compute that
 \begin{equation*}
  \tr(S\wedge T) = (\tr S)T + (\tr T)S - S\circ T - T\circ S .
 \end{equation*}
 The conclusion readily follows. 
\end{proof}

\Cref{compositions,traces} give simple formulas for the curvature of a Riemannian product of spaceforms.

\begin{lemma}
 \label{product}
 Let $(M^m,g)$ and $(N^n,h)$ be Riemannian spaceforms of constant sectional curvature $\mu$ and $\nu$, respectively.
 Set $(X^{m+n},\og):=(M\times N,g + h)$.
 Then
 \begin{align}
  \label{eqn:product_J} J_{\og} & = \frac{m(m-1)\mu + n(n-1)\nu}{2(m+n-1)} , \\
  \label{eqn:product_E} E_{\og} & = \frac{(m-1)\mu - (n-1)\nu}{(m+n)(m+n-2)}\left(ng - mh\right) , \\
  \label{eqn:product_W} W_{\og} & = \frac{\mu+\nu}{2(m+n-1)(m+n-2)}\bigl( n(n-1)g\wedge g \\
   \notag & \qquad\qquad\qquad - 2(n-1)(m-1)g\wedge h + m(m-1)h\wedge h \bigr) ,
 \end{align}
 where $E_{\og} := P_{\og} - \frac{J_{\og}}{m+n}\og$ is the trace-free part of the Schouten tensor of $\og$.
\end{lemma}

\begin{proof}
 Observe that $\Ric_{\og}=(m-1)\mu g + (n-1)\nu h$.
 This readily yields~\eqref{eqn:product_J} and~\eqref{eqn:product_E}.
 Observe also that
 \begin{equation*}
  W_{\og} = \frac{\mu}{2}g\wedge g + \frac{\nu}{2}h\wedge h - \left(E_{\og} + \frac{1}{m+n}J_{\og}\og\right)\wedge(g+h) .
 \end{equation*}
 Combining this with~\eqref{eqn:product_J} and~\eqref{eqn:product_E} yields~\eqref{eqn:product_W}.
\end{proof}

Recall that a Riemannian manifold $(M^n,g)$ is \defn{locally symmetric} if its Riemann curvature tensor is parallel.
In this case, the local invariants~\eqref{eqn:dim6-gci} simplify.

\begin{lemma}
 \label{locally-symmetric-formula}
 Let $(M^6,g)$ be locally symmetric.
 Then
 \begin{align*}
  Q & = -\frac{40}{9}J^3 + 16J\lv E\rv^2 - 16\tr E^3 - 8W_{ijkl}E^{ik}E^{jl} , \\
  L_1 & = 96\lp E, \tr W^2 \rp , \\
  L_2 & = 16\lp E , \tr W^2 \rp - \frac{4}{3}J\tr^2 W^2 , \\
  L_3 & = 4\tr^2 W^3 ,
 \end{align*}
 where $W^2 := W \circ W$ and $W^3 := W \circ W \circ W$.
\end{lemma}

\begin{proof}
 Since $(M^6,g)$ is locally symmetric, it holds that $\nabla W = \nabla P = 0$.
 Hence $B_{ij} = W_{isjt}P^{st}$.
 The conclusion readily follows.
\end{proof}

Specializing \cref{locally-symmetric-formula} to the Riemannian product of two three-dimensional spaceforms and to the Riemannian product of a two- and a four-dimensional spaceform yields enough examples for the results of this note.

\begin{proposition}
 \label{product_33}
 Let $(M^3,g)$ and $(N^3,h)$ be closed three-manifolds with constant sectional curvature $\mu$ and $\lambda$, respectively.
 Then $(\oM^6,\og):=(M \times N, g + h)$ satisfies
 \begin{align*}
  Q & = -\frac{24}{25}(\mu+\lambda)^3 , \\
  L_1 & = 0 , \\
  L_2 & = -\frac{36}{25}(\mu+\lambda)^3 , \\
  L_3 & = \frac{18}{25}(\mu+\lambda)^3 .
 \end{align*}
\end{proposition}

\begin{proof}
 \Cref{product} implies that
 \begin{align*}
  J_{\og} & = \frac{3(\mu+\lambda)}{5} , \\
  E_{\og} & = \frac{\mu - \lambda}{4}\left( g - h\right) , \\
  W_{\og} & = \frac{\mu+\lambda}{20}\left(3g\wedge g - 4g\wedge h + 3h\wedge h\right) .
 \end{align*}
 The conclusion follows from \cref{traces,compositions,locally-symmetric-formula}.
\end{proof}

\begin{proposition}
 \label{product_24}
 Let $(M^2,g)$ be a closed surface with constant sectional curvature $\mu$ and let $(N^4,h)$ be a closed four-manifold with constant sectional curvature $\lambda$.
 Then $(\oM^6,\og):=(M \times N, g + h)$ satisfies
 \begin{align*}
  Q & = -\frac{12}{25}(\mu^3 - 7\lambda\mu^2 + 33\lambda^2\mu - 9\lambda^3) , \\
  L_1 & = 12(\mu+\lambda)^2(\mu-3\lambda) , \\
  L_2 & = \frac{6}{25}(\mu+\lambda)^2(7\mu-33\lambda) , \\
  L_3 & = \frac{39}{25}(\mu+\lambda)^3 .
 \end{align*}
\end{proposition}

\begin{proof}
 \Cref{product} implies that
 \begin{align*}
  J_{\og} & = \frac{\mu + 6\lambda}{5} , \\
  E_{\og} & = \frac{\mu - 3\lambda}{12}\left(2g - h\right) , \\
  W_{\og} & = \frac{\mu+\lambda}{20}\left(6g\wedge g - 3g\wedge h + h\wedge h\right) .
 \end{align*}
 The conclusion follows from \cref{locally-symmetric-formula,compositions,traces}.
\end{proof}

\section{Proofs of main results}
\label{sec:proof}

In this section we prove our main results, \cref{invariant,nonpositive-invariant,fund-group-invariant}.

First we consider the obstruction for a conformally Einstein metric.

\begin{proof}[Proof of \cref{invariant}]
 \Cref{gci} implies that $\mL_1(M,g)$ is conformally invariant.
 Integrating~\eqref{eqn:dim6-pos} yields~\eqref{eqn:invariant}.
 
 Let $g$ be Einstein.
 Equation~\eqref{eqn:invariant} implies that
 \begin{equation*}
  \mL_1(M,g) = \int_M \lv \nabla W\rv^2 \, \dvol .
 \end{equation*}
 Thus $\mL_1(M,g) \geq 0$ with equality if and only if $\nabla W=0$.
 Since $\nabla W = \nabla\Rm$, we see that equality holds if and only if $g$ is locally symmetric.
 
 Suppose now that $\mI$ is a global conformal invariant which is nonnegative at every Einstein metric.
 \Cref{gci} implies that there are constants $a,b,c,e \in \bR$ such that
 \begin{equation}
  \label{eqn:invariant-initial}
  \mI(M,g) = a\mL_1(M,g) + \int_M \left( bQ_6^g + cL_2^g + eL_3^g \right) \, \dvol_g .
 \end{equation}
 We prove that $b = c = e = 0$.
 The conclusion then follows from the existence~\cite{Besse} of closed Einstein manifolds which are not locally symmetric.
 
 First, let $(M^6,g)$ be a closed spaceform with constant sectional curvature $\lambda$.
 \Cref{locally-symmetric-formula} implies that $\mI(M,g) = -120b\lambda^3\Vol_g(M)$.
 Since $\lambda$ is arbitrary, $b=0$.
 
 Second, let $(M^3,g)$ and $(N^3,h)$ be closed spaceforms with constant sectional curvature $\lambda$.
 \Cref{product,product_33} imply that $( M \times N , g + h )$ is Einstein with
 \begin{align*}
  L_1 & = 0 , \\
  L_2 & = -\frac{288}{25}\lambda^3 , \\
  L_3 & = \frac{144}{25}\lambda^3 .
 \end{align*}
 Recalling that $b=0$, we deduce that
 \begin{equation*}
  \mI(M \times N, g + h ) = \frac{144}{25}(e-2c)\lambda^3\Vol_g(M)\Vol_h(N) .
 \end{equation*}
 Since $\lambda$ is arbitrary, $e=2c$.
 
 Third, let $(M^2,g)$ be a closed surface with constant sectional curvature $\lambda$ and let $(N^4,h)$ be a closed four-manifold with constant sectional curvature $\lambda/3$.
 \Cref{product,product_24} imply that $( M \times N , g + h )$ is Einstein with
 \begin{align*}
  L_1 & = 0 , \\
  L_2 & = -\frac{128}{75}\lambda^3 , \\
  L_3 & = \frac{832}{225}\lambda^3 .
 \end{align*}
 Recalling that $b=0$ and $c=2e$, we deduce that
 \begin{equation*}
  \mL(M \times N, g + h) = \frac{64}{225}\lambda^3e\Vol_g(M)\Vol_h(N) .
 \end{equation*}
 Since $\lambda$ is arbitrary, $c=e=0$.
\end{proof}

Next we prove that the nonpositivity of the diffeomorphism invariant $\mL_1(M)$ obstructs the existence of an Einstein metric on a given closed six-manifold $M^6$.

\begin{proof}[Proof of \cref{topological_obstruction}]
 Suppose that $(M^6,g)$ is a closed Einstein six-manifold with $\mL_1(M) \leq 0$.
 \Cref{invariant} implies that $\mL_1(M)=0$ and $g$ is locally symmetric.
 Hence its universal cover $(\cM,\cg)$ is a symmetric Einstein manifold~\cite{Helgason2001}*{Corollary~IV.5.7}.
 
 If $g$ has positive scalar curvature, then Myers' Theorem implies that $\cM$ is compact.
 Hence $\pi_1(M) < \infty$.
 
 If $g$ has nonpositive scalar curvature, then $(\cM,\cg)$ has nonpositive sectional curvature~\cite{Helgason2001}*{Theorem~V.3.1 and Proposition~V.4.2}.
 The Cartan--Hadamard Theorem then implies that $\cM$ is diffeomorphic to $\bR^n$.
\end{proof}

We conclude this section by discussing global conformal invariants which obstruct the existence of an Einstein metric with nonpositive scalar curvature, and hence the existence of an Einstein metric on a closed six-manifold with infinite fundamental group.

\begin{proof}[Proof of \cref{nonpositive-invariant}]
 \Cref{gci} implies that each of $\mQ(M,g)$, $\mL_1(M,g)$, and $\mL_2(M,g)$ is conformally invariant.
 Integrating~\eqref{eqn:dim6-gci} yields~\eqref{eqn:nonpositive-invariant}.
 
 Suppose that $(M^6,g)$ is a closed Einstein six-manifold with $P = \frac{\lambda}{2}g \leq 0$.
 Then
 \begin{align*}
  \mQ(M,g) & = -120\lambda^3\Vol_g(M) , \\
  \mL_1(M,g) & = \int_M \lv\nabla^g W^g\rv_g^2 \, \dvol_g , \\
  \mL_2(M,g) & = -2\lambda\int_M \lv W^g\rv_g^2 \, \dvol_g .
 \end{align*}
 The conclusion readily follows.
\end{proof}

\begin{proof}[Proof of \cref{fund-group-invariant}]
 Let $M^6$ be a closed six-manifold with $\pi_1(M) = \infty$.
 Suppose that $g$ is an Einstein metric on $M$.
 Myers' Theorem implies that $P_g \leq 0$.
 The conclusion follows from \cref{nonpositive-invariant}.
\end{proof}

\section{Some examples}
\label{sec:example}

We conclude by discussing additional properties of the global conformal invariants of \cref{invariant,nonpositive-invariant} through the context of some examples.

Since there are closed six-manifolds which admit geometrically distinct Einstein metrics~\cites{Bohm1998}, one might ask whether $\mL_1(M,g)$ is independent of the choice of Einstein metric.
This is not the case.

\begin{proposition}
 \label{bohm}
 There is an Einstein metric $g$ on $S^6$ with $\mL_1(S^6,g)>0$.
 In particular, $\mL_1(M,g)$ depends on the choice of Einstein metric on $M$.
\end{proposition}

\begin{proof}
 Let $g$ be one of the nontrivial cohomogeneity one Einstein metrics on $S^6$ constructed by B\"ohm~\cite{Bohm1998}*{Theorem~3.6}.
 Then $(S^6,g)$ is not a global Riemannian symmetric space.
 Since $S^6$ is simply connected, $(S^6,g)$ is not locally symmetric.
 We conclude from \cref{invariant} that $\mL_1(S^6,g) > 0$.
 The final conclusion follows from the fact that the round metric $d\theta^2$ on $S^6$ is such that $\mL_1(S^6,d\theta^2)=0$.
\end{proof}

Next we show that $\mQ$, $\mL_1$, and $\mL_2$ give nontrivial obstructions, in the sense that for each invariant one can find a closed Riemannian six-manifold with infinite fundamental group for which the given invariant is negative.

\begin{proposition}
 \label{t2-times-s4-general}
 Let $(T^k,dx^2)$ and $(S^k,d\theta^2)$ denote a $k$-dimensional flat torus and a round $k$-sphere of constant sectional curvature one, respectively.
 Then
 \begin{align*}
  \mQ(S^2 \times T^4, d\theta^2 + dx^2) & = -\frac{48\pi}{25}\Vol_{dx^2}(T^4) , \\
  \mL_1(T^2 \times S^4 , dx^2 + d\theta^2) & = -96\pi^2\Vol_{dx^2}(T^2) , \\
  \mL_2(T^2 \times S^4 , dx^2 + d\theta^2) & = -\frac{528\pi^2}{25}\Vol_{dx^2}(T^2) .
 \end{align*}
 In particular, each of $\mQ$, $\mL_1$, and $\mL_2$ gives a nontrivial obstruction to the existence of an Einstein metric on a closed six-manifold with infinite fundamental group.
\end{proposition}

\begin{proof}
 \Cref{product_24} implies that
 \begin{align*}
  \mL_1(T^2 \times S^4, dx^2 + d\theta^2) & = -36\Vol_{dx^2}(T^2)\Vol_{d\theta^2}(S^4) , \\
  \mL_2(T^2 \times S^4, dx^2 + d\theta^2) & = -\frac{198}{25}\Vol_{dx^2}(T^2)\Vol_{d\theta^2}(S^4) .
 \end{align*}
 and
 \begin{align*}
  \mQ(S^2 \times T^4, d\theta^2 + dx^2) & = -\frac{12}{25}\Vol_{d\theta^2}(S^2)\Vol_{dx^2}(T^4) .
 \end{align*}
 The conclusion follows from the formulas $\Vol_{d\theta^2}(S^2) = 4\pi$ and $\Vol_{d\theta^2}(S^4) = \frac{8}{3}\pi^2$.
\end{proof}

We conclude by pointing out that Einstein metrics are not critical points of $\mL_1$ in general.

\begin{proposition}
 \label{not_critical}
 Let $(S^2,g)$ be the closed two-sphere of constant sectional curvature $3$ and let $(S^4,h)$ be a closed four-sphere of constant sectional curvature $1$.
 Then
 \begin{equation}
  \label{eqn:not_critical}
  \mL_1(S^2\times S^4,g + c^2h) = 128\pi^3c^{-1}(3c+1)^2(c-1)
 \end{equation}
 for all $c>0$.
 In particular, the Einstein manifold $(S^2\times S^4,g + h)$ is not a critical point of $\mL_1$ in the space of Riemannian metrics on $S^2\times S^4$.
\end{proposition}

\begin{proof}
 Equation~\eqref{eqn:not_critical} follows immediately from \cref{product_24} and the fact that $(S^4,ch)$ is a spaceform with constant sectional curvature $c^{-1}$.
 It is clear that $(S^2\times S^4,g + h)$ is Einstein, while~\eqref{eqn:not_critical} implies that
 \begin{equation*}
  \left.\frac{d}{dc}\right|_{c=1} \mI(S^2\times S^4,g + ch) \not= 0 . \qedhere
 \end{equation*}
\end{proof}

\section*{Acknowledgments}
I thank Misha Gromov and Claude LeBrun for helpful comments and encouragement.
This work was partially supported by the Simons Foundation (Grant \#524601).

\bibliographystyle{abbrv}
\bibliography{bib}
\end{document}